\documentclass[11pt]{amsart}
\usepackage{fullpage}

\usepackage{amssymb}

\usepackage[english]{babel}
\theoremstyle{plain}
\newtheorem{theorem}                {Theorem}      [section]
\newtheorem{proposition}  [theorem]  {Proposition}

\newtheorem{lemma}        [theorem]  {Lemma}

\theoremstyle{definition}

\newtheorem{remark}       [theorem]  {Remark}

\numberwithin{equation}{section}

\def \R{{\mathbb R}}

\numberwithin{equation}{section}

\begin{document}

\title[Equivariant biharmonic maps]{Rotationally symmetric biharmonic maps between models}

\author{S.~Montaldo}
\address{Universit\`a degli Studi di Cagliari\\
Dipartimento di Matematica e Informatica\\
Via Ospedale 72\\
09124 Cagliari, Italia}
\email{montaldo@unica.it}

\author{C.~Oniciuc}
\address{Faculty of Mathematics\\ ``Al.I. Cuza'' University of Iasi\\
Bd. Carol I no. 11 \\
700506 Iasi, ROMANIA}
\email{oniciucc@uaic.ro}

\author{A.~Ratto}
\address{Universit\`a degli Studi di Cagliari\\
Dipartimento di Matematica e Informatica\\
Via Ospedale 72\\
09124 Cagliari, Italia}
\email{rattoa@unica.it}

\begin{abstract}
The main aim of this paper is to study existence and stability properties of rotationally symmetric proper biharmonic maps
between two $m$-dimensional models (in the sense of Greene and Wu).
 We obtain a complete classification of rotationally symmetric, proper biharmonic conformal diffeomorphisms in the special case that $m=4$ and the models have constant sectional curvature. Then, by introducing the Hamiltonian associated to this problem, we also obtain a complete description of conformal proper biharmonic solutions in the case that the domain model
is $\R^4$.  In the second part of the paper we carry out a stability study with respect to equivariant variations (equivariant stability). In particular, we prove that: (i)  the inverse of the stereographic projection from the open $4$-dimensional Euclidean ball to the hyperbolic space is equivariant stable; (ii) the inverse of the stereographic projection from the closed $4$-dimensional Euclidean ball to the sphere is equivariant stable with respect to variations which preserve the boundary data.

\end{abstract}

\thanks{Corresponding author e-mail: montaldo@unica.it\\
Work supported by P.R.I.N. 2010/11 -- Variet\`a reali e complesse: geometria, topologia e analisi armonica -- Italy, and G.N.S.A.G.A., INdAM, Italy;
Romanian National Authority for Scientific Research, CNCS -- UEFISCDI, project number PN-II-ID-PCE-2012-4-0640. This paper is also supported by Contributo d'Ateneo, University of Cagliari, Italy.}



\maketitle

\section{Introduction}\label{intro}
{\it Harmonic maps}  are critical points of the {\em energy} functional
\begin{equation}\label{energia}
E(\varphi)=\frac{1}{2}\int_{M}\,|d\varphi|^2\,dv_g \,\, ,
\end{equation}
where $\varphi:(M,g)\to(N,h)$ is a smooth map between two Riemannian
manifolds $M$ and $N$. In analytical terms, the condition of harmonicity is equivalent to the fact that the map $\varphi$ is a solution of the Euler-Lagrange equation associated to the energy functional \eqref{energia}, i.e.
\begin{equation}\label{harmonicityequation}
    {\rm trace} \, \nabla d \varphi =0 \,\, .
\end{equation}
The left member of \eqref{harmonicityequation} is a vector field along the map $\varphi$, or, equivalently, a section of the pull-back bundle $\varphi^{-1} \, (TN)$: it is called {\em tension field} and denoted $\tau (\varphi)$. In local charts, the tension field is given by the following expression:
\begin{equation}\label{tensionfield}
    \tau^\gamma (\varphi)= g^{ij} \left ( \nabla(d \varphi) \right )_{ij}^\gamma \,\, ,
\end{equation}
where
\begin{equation}\label{secondaformalocale}
    \left ( \nabla(d \varphi) \right )_{ij}^\gamma = \frac{\partial^2 \varphi^\gamma}{\partial x^i \, \partial x^j} - {}^M\Gamma_{ij}^k \, \frac{\partial \varphi^\gamma}{\partial x^k} \, + \, {}^N\Gamma_{\alpha \beta}^\gamma \, \frac{\partial \varphi^\alpha}{\partial x^i} \, \frac{\partial \varphi^\beta}{\partial x^j} \,\, .
\end{equation}
In \eqref{secondaformalocale}, ${}^M\Gamma$ and ${}^N\Gamma$ denote the Christoffel symbols of the Levi-Civita connections of $(M,g)$ and $(N,h)$ respectively. Also, note that Einstein's convention of sum over repeated indices is adopted.
We refer to \cite{PBJCW,JEELLE2,JEELLE1,Xin} for notation and
background on harmonic maps.

A related topic of growing interest deals with the study of the so-called {\it biharmonic maps}: these maps, which provide a natural generalisation of harmonic maps, are the critical points of the bienergy functional (as suggested by Eells--Lemaire \cite{EL83})
\begin{equation}\label{bienergia}
    E_2(\varphi)=\frac{1}{2}\int_{M}\,|\tau (\varphi)|^2\,dv_g \,\, .
\end{equation}
In \cite{Jiang} Jiang derived the first variation and the second variation formulas for the bienergy. In particular, he showed that the Euler-Lagrange equation associated to $E_2(\varphi)$ is
\begin{equation}\label{bitensionfield}
    \tau_2(\varphi) = - J\left (\tau(\varphi) \right ) = - \triangle \tau(\varphi)- \rm{trace} R^N(d \varphi, \tau(\varphi)) d \varphi = 0 \,\, ,
    \end{equation}
where $J$ is (formally) the Jacobi operator of $\varphi$, $\triangle$ is the rough Laplacian defined on sections of $\varphi^{-1} \, (TN)$ and
\begin{equation}\label{curvatura}
    R^N (X,Y)= \nabla_X \nabla_Y - \nabla_Y \nabla_X -\nabla_{[X,Y]}
\end{equation}
is the curvature operator on $(N,h)$.
Equation \eqref{bitensionfield} is a {\it fourth order}  semi-linear elliptic system of differential equations. We also note that any harmonic map is an absolute minimum of the bienergy, and so it is trivially biharmonic; thus, a general working plan is to study the existence of {\it proper} biharmonic maps, i.e., biharmonic maps which are not harmonic. We refer to \cite{Montaldo} for existence results and general properties of biharmonic maps.

{\it Equivariant theory} deals with special families of maps having enough symmetries  to guarantee that harmonicity reduces to the study of a second order {\it ordinary}
differential equation (we refer to \cite{BR, CMOP, EELRAT, MR, RR, Xin} for background and examples). In \cite{MR}, we developped a systematic approach to equivariant theory for biharmonic maps (see \cite{HM}, for recent developments). In this framework, the aim of this work is to study rotationally symmetric biharmonic maps between two $m$-dimensional models (in the sense of \cite{GW}). The study of  rotationally symmetric biharmonic maps has been started in \cite{Montaldobis} and, in the particular case of maps between surfaces, in \cite{WangOuYang}.

In this paper we obtain a complete classification of  proper biharmonic conformal diffeomorphisms between two $m$-dimensional models in the special case that $m=4$ and the models have constant sectional curvature (Proposition~\ref{proposizionediclassificazione}). Then, by introducing the Hamiltonian associated to this problem, we obtain a complete classification of conformal proper biharmonic maps in the case that the domain model is $\R^4$ and the target is any $4$-dimensional model (Theorem~\ref{maintheorem}).

In the second part of the paper we carry out a stability study with respect to equivariant variations (equivariant stability). In particular, we prove that: (i)  the inverse of the stereographic projection from the open $4$-dimensional Euclidean ball to the hyperbolic space is equivariant stable (Theorem~\ref{teoremavariazionesecondaiperbolico}); (ii) the inverse of the stereographic projection from the closed $4$-dimensional Euclidean ball to the sphere is equivariant stable with respect to variations which preserve the boundary data (Theorem~\ref{teoremavariazionesecondaDirichlet}).

\begin{remark}\label{remarksullevariazionisuicompattiD} We stated above that equations \eqref{harmonicityequation} and \eqref{bitensionfield} are the Euler-Lagrange equations associated with the energy functional \eqref{energia} and the  bienergy functional \eqref{bienergia} respectively. We point out that, in the case that the domain manifold $M$ is not compact, the previous claim can be formally verified by considering these functionals over compact domains $D \subset M$, with smooth boundary $\partial D$, and smooth variations $\varphi_t$ which coincide with $\varphi$ on $M \,\backslash\,{\rm int}(D)$. In this context, a map $\varphi$ is a critical point if it is such over any compact domain $D$ in $M$ (see \cite{PBJCW} for details).
\end{remark}

{\bf Acknowledgement}.
The authors wish to thank the referee for some useful comments that have improved the quality of the paper and, in particular, for pointing out the solutions given in Remark~\ref{re:referee}.

\section{Notation and preliminaries}\label{sezvarapproach}

By definition, a point $o$ of a Riemannian manifold $M$ is said to be a \emph{pole} if the exponential map $exp_o \, : \, T_oM \rightarrow M$ is a diffeomorphism. Following \cite{GW}, we say that an $m$-dimensional manifold $(M^m(o),\,g)$ with a pole $o$ is a \emph{model} if and only if every linear isometry of $T_oM$ can be realized as the differential at $o$ of an isometry of $M$. The starting goal of Greene and Wu was to develop a good function  theory on Cartan-Hadamard manifolds (CH-manifolds), i.e., complete simply-connected Riemannian manifolds of nonpositive sectional curvature. Models are more general than CH-manifolds (for instance, their curvature can be of variable sign) and turned out to be a suitable environment for Laplacian and Hessian comparison theorems. A significant geometric property of a model is the fact that we can describe it by means of geodesic polar coordinates centered at the pole $o$, as follows:
\begin{equation}\label{model}
(M^m(o),\,g) = \left ( \,S^{m-1} \times [0,\, + \infty) , \, f^2(r)\, g_{S^{m-1}} \, + \, dr^2 \,\right ) \,\, ,
\end{equation}
where
$$ (\,S^{m-1},\, g_{S^{m-1}} \, )
$$
denotes the $(m-1)$-dimensional Euclidean unit sphere, and the function $f(r)$ is a smooth function which satisfies
\begin{equation}\label{condizioni-su-f}
f(0)=0 \,, \quad f'(0)=1 \quad {\rm and}\quad f(r)>0 \quad {\rm if} \,\, r>0 \,\, .
\end{equation}
We also note that $r$ measures the geodesic distance from the pole $o$. To shorten notation and emphasize the role of the function $f$, we shall write $M_f^m(o)$ to denote a model as in \eqref{model}.

\begin{remark}\label{remarksuR^meH^m} We observe that, if $f(r)=r$, then $M_f^m(o)=\R^m$. We also point out that, if $f(r)= (1\slash c) \, \sinh (r \,c)$ ($c>0$), then $M_f^m(o)$ represents $H^m(-c^2)$, i.e., the $m$-dimensional hyperbolic space of constant sectional curvature $-\, c^2$.
\end{remark}

For future use (see \cite{GW}), we also recall that the radial curvature $K(r)$ ($r>0$) of a model $M_f^m(o)$ is defined as the sectional curvature of any plane which contains $\partial \, \slash\, \partial \, r$ . The radial curvature is related to the function $f(r)$ by means of the following fundamental equation (the Jacobi equation):
\begin{equation}\label{equazionecurvaturaradiale}
    f''(r)\,+\,K(r)\,f(r)\,=\,0 \,\, ,\,\, f(0)=0\,, \,\,f'(0)=1 \,\, .
\end{equation}

The aim of this paper is to study rotationally symmetric maps between two models, i.e. maps of the following type:
\begin{eqnarray}\label{rotationallysymmetricmaps}
 \nonumber
    \varphi_{\alpha}\,: \,\,M_f^m(o) &\to&
    \,\,M' _h \,^m (o') \\
    ( \theta, \, r) \, &\mapsto& \, (\theta, \, \alpha(r)) \,\, ,
\end{eqnarray}
where the function $\alpha(r)$ is smooth on $[0,\,+\infty)$, positive on $(0,\,+\infty)$  and, to ensure continuity, satisfies the boundary condition
\begin{equation}\label{boundaryconditionforalfa}
    \alpha(0)=0 \,\, .
\end{equation}
\begin{remark}\label{remarksuS^m}
With a slight abuse of terminology, in our study of rotationally symmetric biharmonic maps we shall also consider the case that $f(r)$ is defined on a finite interval $[0,\,b]$, with $f(b)=0$ and $f'(b)=-1$. In particular, we shall pay special attention to the case $f(r)= (1 \slash d)\, \sin (r \, d)$, where $d>0$ and $\, 0 \leq r \leq (\pi \slash d )\,$: in this case, our manifold is the Euclidean $m$-sphere $S^m(d^2)$ of constant sectional curvature $d^2$.
\end{remark}
\begin{remark} More generally, one could apply the reduction technique of this paper to the following class of equivariant maps:
\begin{eqnarray}\label{equivariantsymmetricmaps}
\nonumber
    \varphi_{\alpha}\,: \,\,M_f^m(o) &\to&
    \,\,M' _h\,^n (o') \\
    ( \theta, \, r) \, &\mapsto& \, (\Psi_\lambda(\theta), \, \alpha(r)) \,\, ,
\end{eqnarray}
 where $\Psi_\lambda(\theta)$ is a so-called {\it eigenmap} of eigenvalue $\lambda$ . That means that
$\Psi_\lambda:\, S^{(m-1)}\to S^{(n-1)}$ is a harmonic map with {\it constant} energy density equal to $(\lambda \slash 2)$. Important examples of eigenmaps are: the identity map of $S^{(m-1)}$ ($\lambda=m-1$), the k-fold rotation $e^{i\theta} \rightsquigarrow e^{ik\theta}$ of $S^1$ ($\lambda=k^2$); and, also, the Hopf fibrations $S^3 \to S^2$, $S^7 \to S^4$ and $S^{15} \to S^8$, with $\lambda$ equal to 8, 16 and 32 respectively. However, since the primary aim of this paper is the study of conformal diffeomorphisms, we decided to focus on \eqref{rotationallysymmetricmaps}, i.e., on the case that $\Psi_\lambda(\theta)$ is the identity map.
\end{remark}

By way of summary, the unknown function $\alpha(r)$ in \eqref{rotationallysymmetricmaps} has to be determined in such a way that $\varphi_{\alpha}$ be a biharmonic map. To this purpose, a calculation based on \eqref{secondaformalocale} shows that the tension field of a rotationally symmetric map as in \eqref{rotationallysymmetricmaps} is given by
\begin{equation}\label{tensionfieldtrawarped}
    \tau (\varphi_{\alpha})= \left [\ddot{\alpha}(r) + (m-1)\, \frac{f'(r)}{f(r)}\, \dot{\alpha}(r)- (m-1)\, \, \frac{h(\alpha)\,h'(\alpha)}{f^2(r)} \right ] \,\, \frac{\partial}{\partial \alpha} \,\,.
\end{equation}
Therefore, in this case the reduced bienergy is given, up to an irrelevant constant factor, by the following expression:
\begin{equation}\label{reducedbienergiatrawarped}
    E_2(\varphi_\alpha) =  \frac{1}{2}\, \int_0^{+\infty } \left [\ddot{\alpha}(r) + (m-1)\frac{f'(r)}{f(r)}\, \dot{\alpha}(r)- (m-1)\frac{h(\alpha)\,h'(\alpha)}{f^2(r)} \right ]^2 f^{m-1}(r)\, dr\,\,.
\end{equation}
We observe that the reduced bienergy functional \eqref{reducedbienergiatrawarped} is of the form
\begin{equation}\label{bienergiaridotta}
    E_2(\varphi_\alpha)= \int_0^{+\infty} \,\, L(r, \alpha, \dot{\alpha}, \ddot{\alpha}) \, dr \,\, ,
\end{equation}
for a suitable Lagrangian function $L$. We showed in \cite{MR} that, in this symmetric context, the condition of biharmonicity is
\begin{equation}\label{eqeulero2nonesplicita}
     \frac{\partial L}{\partial \alpha} - \frac{d}{dr} \, \left ( \frac{\partial L}{\partial \dot{\alpha}} \right ) + \frac{d\,^2}{dr^2} \, \left ( \frac{\partial L}{\partial \ddot{\alpha}} \right )  \, = \, 0 \,\, .
\end{equation}
Now, a straightforward computation leads us to the explicit expression of \eqref{eqeulero2nonesplicita} in our context (to simplify notation, we write $f$ and $h(\alpha)$ instead of $f(r)$ and $h(\alpha(r))$ respectively):
\begin{equation}\label{eqeulero2}
\begin{aligned}
   & f^{m-5} \Big ((m-1) h(\alpha ) (2 f f''
   h'(\alpha )-2 (m-3) f f' \dot{\alpha} h''(\alpha
   )+2 (m-4) f'^2 h'(\alpha )\\
   &-f^2
   (h^{(3)}(\alpha ) \dot{\alpha}^2+2\ddot{\alpha}
   h''(\alpha ))+(m-1) h'(\alpha )^3)+f
   ((m-3) (m-1) f f'^2 \ddot{\alpha}\\
   &-(m-3) (m-1)
   f'^3 \dot{\alpha}+(m-1) f (f (f^{(3)}
   \dot{\alpha}+2 f'' \ddot{\alpha})-2 \ddot{\alpha}
   h'(\alpha )^2\\
   &-3 \dot{\alpha}^2 h'(\alpha ) h''(\alpha
   ))+(m-1) f' (\dot{\alpha} ((m-4) f
   f''-2 (m-3) h'(\alpha )^2)+2 f^2 \alpha
   ^{(3)})\\
   &+f^3 \alpha ^{(4)})+(m-1)^2
   h(\alpha )^2 h'(\alpha ) h''(\alpha )\Big )=0\,\, .\\
\end{aligned}
\end{equation}

\begin{remark}
We point out that, to  the purpose of comparison  with the equation
given, for $m=2$, in \cite[Corollary~2.3]{WangOuYang}, equation \eqref{eqeulero2}  can be rewritten as follows:
$$
\begin{cases}
F''+(m-1) \dfrac{f f' F'- h'(\alpha)^2 F}{f^2}-(m-1) \dfrac{h(\alpha) h''(\alpha)F}{f^2}=0\\
F=\ddot{\alpha}+(m-1) \dfrac{f'}{f} \dot{\alpha} -(m-1) \dfrac{h(\alpha) h'(\alpha)}{f^2}\,.
\end{cases}
$$
\end{remark}
In summary, a rotationally symmetric map $\varphi_{\alpha}$ as in \eqref{rotationallysymmetricmaps} is \emph{biharmonic} if and only if $\alpha$ is a solution of \eqref{eqeulero2}.

In some cases it is useful to associate to a Lagrangian $L$ as in \eqref{bienergiaridotta} its corresponding Hamiltonian $H$, which is defined as follows (see, for instance, \cite{DNF}):
\begin{equation}\label{definizionehamiltoniana}
    H(r, \alpha, \dot{\alpha}, \ddot{\alpha})=\dot{\alpha}\, \left ( \frac{\partial L}{\partial \dot{\alpha}} -\frac{d\,}{dr} \, \frac{\partial L}{\partial \ddot{\alpha}}
\right ) \,+\, \ddot{\alpha}\, \left ( \frac{\partial L}{\partial \ddot{\alpha}} \right )\,- \, L \,\, .
\end{equation}
An useful feature of the Hamiltonian $H$ is the following well-known property:
\begin{proposition}\label{propositionH} Assume that the Lagrangian $L$ in \eqref{bienergiaridotta} does not depend on $r$ (i.e., $L=L( \alpha, \dot{\alpha}, \ddot{\alpha})$). Then $H$ is constant along any solution $\alpha$ of \eqref{eqeulero2nonesplicita}.
\end{proposition}
The verification of the previous proposition amounts to a straightforward, direct computation of $\dot{H}$ along $\alpha$, so we omit the details. In this case, we say that $H$ is a \emph{prime integral} of the biharmonicity equation.

\section{Biharmonic, rotationally symmetric, conformal maps between models of constant sectional curvature}\label{conformal-maps}

The difficulty of the general problem (fourth order equations) suggests to restrict investigation to speficic, geometrically significant, families of maps. In particular, in this section we look for  rotationally symmetric, proper biharmonic \emph{conformal} diffeomorphisms between $m$-dimensional models of constant sectional curvature. We shall be able to obtain a complete description of such maps.

First, let us point out when a rotationally symmetric map as in \eqref{rotationallysymmetricmaps} is conformal. Comparing dilations of vectors which are respectively orthogonal and tangent to the radial direction, we easily find that a map of the type \eqref{rotationallysymmetricmaps} is conformal iff
$$
    \dot{\alpha}^2=\frac{h^2(\alpha)}{f^2(r)} \,\,
    $$
or, taking into account the boundary condition \eqref{boundaryconditionforalfa},
\begin{equation}\label{equazionediconformalita}
    \dot{\alpha}=\frac{h(\alpha)}{f(r)} \,\, .
\end{equation}
Using \eqref{equazionediconformalita} into the biharmonicity equation \eqref{eqeulero2}, we find that the condition for biharmonicity of a conformal map of the type \eqref{rotationallysymmetricmaps} is the following:

\begin{equation}\label{equazionebiarmonicaeconformeperognim}
   \begin{aligned}
& (m-2) f^{m-5} h(\alpha )\Big(f^2 f^{(3)}+h'(\alpha
   ) \big(4 f f''+(m-5) h(\alpha ) h''(\alpha
   )\big)\\
   &+(3 m-14) f'^2 h'(\alpha )-2 (m-4) f'^3+\\
   &f'
   \big((m-7) f f''-2 (m-4) h(\alpha ) h''(\alpha )-2
   (m-4) h'(\alpha )^2\big)-\\
   & h(\alpha )^2 h^{(3)}(\alpha
   )+(m-2) h'(\alpha )^3\Big) = 0\,\, .
\end{aligned}
\end{equation}
In particular, if $m=4$, equation \eqref{equazionebiarmonicaeconformeperognim} becomes:
\begin{equation}\label{equazionebiarmonicaeconformem=4}
\begin{aligned}
&\frac{2 h(\alpha)}{f} \,\,  \Big(f^2 f^{(3)}+h'(\alpha)
   \big(4 f f''-h(\alpha) h''(\alpha)\big) \\
   & -2
   f'^2 h'(\alpha )-3 f f' f''-h(\alpha )^2
   h^{(3)}(\alpha)+2 h'(\alpha)^3\Big) =0 \,\, .
   \end{aligned}
\end{equation}

Now, we are in the position to provide a complete description of solutions in the case of maps between 4-dimensional models of constant sectional curvature. More precisely, we have:

\begin{proposition}\label{proposizionediclassificazione} Let us consider rotationally symmetric maps as in \eqref{rotationallysymmetricmaps}-\eqref{boundaryconditionforalfa}. Let us assume that $m=4$ and that both models have constant sectional curvature (see Remarks \ref{remarksuR^meH^m} and \ref{remarksuS^m}). Then the  biharmonic conformal diffeomorphisms of type  \eqref{rotationallysymmetricmaps}-\eqref{boundaryconditionforalfa} can be enumerated as follows ($c,\,d$ denote real positive constants):
\begin{enumerate}
\item[\textbf{Case 1}] - $f(r)=r\,.$
\begin{enumerate}
\item[A] - $h(\alpha)=\alpha$, $\alpha(r)= c\,r$ (harmonic diffeomorphisms from $\R^4$ to itself);
\item[B]- $h(\alpha)=(1/d)\, \sin(d\, \alpha)$ and
\begin{equation}\label{definizionesoluzionedaeuclideoasfera}
    \alpha(r)= \,\,\frac{2}{d} \, \arctan (c^2\,r)\, ,\,\,
\end{equation}
(proper biharmonic diffeomorphisms from $\R^4$ to $S^4(d^2)\smallsetminus\{{\rm south \,\, pole}\}$);
\item[C]- $h(\alpha)=(1/d )\,\sinh(d\, \alpha)$ and
\begin{equation}\label{definizionesoluzionedaeuclideoaiperbolico}
    \alpha(r)= \,\, \,\,\frac{2}{d} \, \tanh ^{-1} (c^2\,r)\,\,\,\, (0 \leq r < \frac{1}{c^2})
\end{equation}
(proper biharmonic diffeomorphisms from $B^4(1 \slash c^2)$ (i.e., the open ball of radius $(1 \slash c^2)$ in $\R^4$) to $H^4(-d^2)$).
\end{enumerate}
\item[\textbf{Case 2}] -  $f(r)=(1/c)\, \sin(c \, r).$
\begin{enumerate}
\item[A] - if $h(\alpha)=\alpha$, then there is no solution;
\item[B]-  if $h(\alpha)=(1/d)\, \sin(d\,  \alpha)$, then
$$
\alpha (r)= (c \slash d)\,r\,\,(0\leq r \leq \frac{\pi}{c})
$$
gives rise to a harmonic conformal diffeomorphism, but in this case we do not have proper biharmonic examples;
\item[C]-  if $h(\alpha)=(1/d)\, \sinh(d \,\alpha)$, then there is no solution.
\end{enumerate}
\item[\textbf{Case 3}] -  $f(r)=(1/c) \sinh(c\, r).$
\begin{enumerate}
\item[A] - if $h(\alpha)=\alpha$, then there is no solution;
\item[B]-  if $h(\alpha)=(1/d)\, \sin(d \,\alpha)$, then there is no solution;
\item[C]-  if $h(\alpha)=(1/d)\, \sinh(d \,\alpha)$, then
$$
\alpha (r)= (c \slash d)\,r \,\, (r \geq 0),
$$
 produces a harmonic conformal diffeomorphism, while there is no proper biharmonic example.
 \end{enumerate}
 \end{enumerate}
\end{proposition}
\begin{proof}\label{proofofproposizionediclassificazione} The proof amounts to a case by case analytical inspection of \eqref{equazionebiarmonicaeconformem=4}. Indeed,\\
\textbf{Case 1 - $f(r)=r\,.$} In this case, equation \eqref{equazionebiarmonicaeconformem=4} becomes (here $h=h(\alpha)$):
$$
 2\,h'^3-h'\,(2+h\,h'')\,-\, h^2\, h''' \,=\,0
$$
which is satisfied for

A - $h(\alpha)=\alpha$ (harmonic maps);

B - $h(\alpha)=(1/d)\, \sin(d\, \alpha)$: in this case we have proper biharmonic diffeomorphisms whose explicit expression \eqref{definizionesoluzionedaeuclideoasfera} for $\alpha(r)$ can be derived directly by the condition of conformality \eqref{equazionediconformalita});

C - $h(\alpha)=(1/d)\,\sinh(d\, \alpha)$, which again produces proper biharmonic diffeomorphisms (also in this case, the explicit expression \eqref{definizionesoluzionedaeuclideoaiperbolico} for $\alpha(r)$ can be derived directly by the condition of conformality \eqref{equazionediconformalita}). It is important to point out that, in this case, solutions are defined on a finite interval only.\\
\textbf{Case 2 -  $f(r)=(1/c)\, \sin(c \, r)\,.$} The biharmonicity condition \eqref{equazionebiarmonicaeconformem=4} becomes

\begin{equation}\label{eq:33caso2}
\begin{aligned}
2 h'(\alpha ) \left(\cos (2 c \,r)-h(\alpha )
   h''(\alpha )-3\right)+2 \sin (c\, r) \sin (2 c\, r)
   -2 h(\alpha )^2
   h^{(3)}(\alpha )+4 h'(\alpha )^3=0 \,\,.
\end{aligned}
\end{equation}

Therefore:

A - if $h(\alpha)=\alpha$,  the left-hand side of \eqref{eq:33caso2} becomes
$$
-8 \sin ^2\left(\frac{c\, r}{2}\right) \sin ^2(c \,r) \,\,,
$$
which cannot vanish identically, so that there is no solution in this case.

B -  if $h(\alpha)=(1/d)\, \sin(d\, \alpha)$  \eqref{eq:33caso2} becomes
$$
4 \sin ^2(c \,r) (\cos (c\, r)-\cos (d\,\alpha))=0 \,\,,
$$
so that in this case we have harmonic diffeomorphisms, but not proper biharmonic examples.

C -  if $h(\alpha)=(1/d)\, \sinh(d \,\alpha)$ the left-hand side of \eqref{eq:33caso2} becomes
$$
4 \sin ^2(c\, r) (\cos (c \,r)-\cosh (d\,\alpha ))\,\,.
$$
Now, since $\cosh x \geq 1$ for all $ x \,\in \, \R$, it is obvious that the latter
cannot vanish identically, thus there is no solution in this case too.\\
\textbf{Case 3 -  $f(r)=(1/c) \sinh(c \,r)\,.$ }The biharmonicity condition becomes
\begin{eqnarray}\label{eq:33caso3}
-2 \left(\sinh (c\, r) \sinh (2 \,c
   r)+h^{(3)}(\alpha ) h(\alpha )^2\right)\nonumber\\
    +2 h'(\alpha )
   \left(\cosh (2 c \,r)-h(\alpha ) h''(\alpha )-3\right)+4
   h'(\alpha )^3=0 \,\, .
\end{eqnarray}

Now,

A - if $h(\alpha)=\alpha$, the left-hand side of \eqref{eq:33caso3} becomes
$$
-8 \sinh ^2\left(\frac{c \,r}{2}\right) \sinh ^2(c\, r)\,\,,
$$
from which we deduce nonexistence in this case.

B -  if $h(\alpha)=(1/d)\, \sin(d\, \alpha)$, \eqref{eq:33caso3} becomes
$$
4 \sinh ^2(c r) (\cos (d\,\alpha  )-\cosh (c\, r))=0 \,\,.
$$
As in Case 2C above, we easily conclude that there is no nontrivial solution.

C -  if $h(\alpha)=(1/d)\, \sinh(d\, \alpha)$, \eqref{eq:33caso3} becomes
$$
-4 \sinh ^2(c \,r) (\cosh (c\, r)-\cosh (d\,\alpha))=0 \,\,,
$$
which simply yields harmonic diffeomorphisms, but not proper biharmonic examples.
\end{proof}

\begin{remark}
We point out that \eqref{definizionesoluzionedaeuclideoasfera} and \eqref{definizionesoluzionedaeuclideoaiperbolico} correspond to a dilation of $\R^4$ composed with the inverse of the stereographic projection of $S^4(d^2)$ and $H^4(-d^2)$ respectively.
\end{remark}

\begin{remark} A direct, case by case inspection shows that, if $m>2$ and $m\neq4$, then there exists no rotationally symmetric, conformal, proper biharmonic map between $m$-dimensional models of constant sectional curvature. By way of example, if $f(r)=r$ and $h(\alpha)= \sin \alpha$, then condition \eqref{equazionebiarmonicaeconformeperognim} becomes:
\begin{equation}\label{casogeneraleconformebiharmonicesempio}
    4\,(m-2)\,(m-4)\,r^{m-5}\,\sin (2\,\alpha) \,\sin^4 (\alpha\slash2) \,=\,0 \, \, ,
\end{equation}
which is not possible when $m \neq 2,\,4$. The other cases are similar, so we omit further details.
\end{remark}

\begin{remark}\label{Jager-Kaul} It is interesting to compare the proper biharmonic diffeomorphisms from $\R^4$ to $S^4 \smallsetminus \{{\rm south \,\, pole} \}$ with the qualitative behaviour of rotationally symmetric harmonic maps (see \cite{JK}). In particular, J\"{a}ger and Kaul proved that the image of the functions $\alpha(r)$ associated to these harmonic maps cover a range $[0,\,R_4]$, with $(\pi \slash 2) < R_4 < \pi$, and $\alpha(r)$ oscillates around $(\pi \slash 2)$ as $r \rightarrow +\infty$ (further details concerning the numerical value of $R_4$ can be found in \cite{JK}).
J\"{a}ger and Kaul applied their results to draw some interesting conclusions concerning the existence of rotationally symmetric solutions to the Dirichlet problem for maps from the Euclidean unit $m$-ball $B^m$ to $S^m$. In particular, if $m=4$, they proved that the Dirichlet problem with boundary data
\begin{equation}\label{dirichletboundarydata}
    (\theta,\,1) \rightarrow (\theta,\,\alpha(1)=R^*)
\end{equation}
admits a rotationally symmetric solution if and only if:
\begin{equation}\label{soluzionijagerkaul}
    0 \leq R^* \leq R_4 \,\,.
\end{equation}
By contrast, the existence of the  proper biharmonic conformal diffeomorphisms of Proposition \ref{proposizionediclassificazione} implies that the boundary value problem \eqref{dirichletboundarydata} admits proper biharmonic solutions for all
$$
0 \leq R^* < \pi \,\,.
$$
\end{remark}

\begin{remark}\label{lavoroLoubeaueBaird} The explicit solutions of Cases 1B and 1C above were also studied in \cite{Baird-et-al} and \cite{Loubeau-et-al}. In particular, it was shown in \cite{Loubeau-et-al} that they do not provide examples of biharmonic morphisms.
\end{remark}

\begin{remark}\label{re:referee}
If, in the context of Proposition~\ref{proposizionediclassificazione}, we admit solutions which are not continuous at the pole (i.e. which do not satisfy the boundary condition \eqref{boundaryconditionforalfa}), then we have some further examples of proper biharmonic conformal diffeomorphisms associated with the choice
\begin{equation}\label{equazionediconformalita-bis}
    \dot{\alpha}=-\frac{h(\alpha)}{f(r)}
\end{equation}
in place of  \eqref{equazionediconformalita}. More precisely, by using similar arguments, it is not difficult to check that, in Case 1 ($f(r)=r$), we have:

A - if $h(\alpha)=\alpha$, then $\alpha(r)=c/r$ provides the proper biharmonic, conformal inversion $\varphi:\R^4\setminus\{0\}\to \R^4\setminus\{0\}$ given by
$\varphi(x)=cx/|x|^2$.

B - if $h(\alpha)=(1/d)\, \sin(d\, \alpha)$, then $\alpha(r)=(2/d) \arctan(c^2/r)$ gives rise to a family of proper biharmonic conformal diffeomorphisms $\varphi:\R^4\setminus\{0\}\to S^4(d^2)\setminus\{\text{two poles}\}$, obtained by composing the inversion (we refer to A above) with the inverse of stereographic projection. Similarly,

C if $h(\alpha)=(1/d)\,\sinh(d\, \alpha)$, then $\alpha(r)=(2/d) \tanh^{-1}(c^2/r)$ gives rise to  proper biharmonic conformal diffeomorphisms $\varphi:\R^4\setminus\{\overline{B^4 (c^2)}\}\to H^4(-d^2)\setminus\{\text{pole}\}$.

In the remaining cases ($f(r)=(1/c) \sin(c \,r)$ or $f(r)=(1/c) \sinh(c \,r)$), an analysis similar to that in the proof of Proposition~\ref{proposizionediclassificazione} shows that there are no examples of proper biharmonic conformal diffeomorphisms
which satisfy \eqref{equazionediconformalita-bis}.

\end{remark}

\section{Further developments: conformal, proper biharmonic maps from $\R^m$ to an $m$-dimensional model}

A natural, general development of the work which we carried out in Section \ref{conformal-maps} is to study when a \emph{conformal}, rotationally symmetric map as in \eqref{rotationallysymmetricmaps} is proper biharmonic.

Since the general case appears to be difficult, we start with the case that the domain model is the Euclidean space. In particular,
using the Hamiltonian function $H$ defined in \eqref{definizionehamiltoniana}, we shall obtain a complete answer in the case that the domain model is $\R^4$.

In order to simplify the analysis, it is convenient to perform the following change of variable:
\begin{equation}\label{cambiovariabile euclideanspaces}
    r=e^t \, , \quad t \in \R \, , \qquad \beta(t)= \alpha(e^t) \, \, ,
\end{equation}
so that the boundary condition \eqref{boundaryconditionforalfa} becomes
\begin{equation}\label{boundaryconditionforbeta}
    \lim_{t\rightarrow -\,\infty} \, \beta (t) = 0 \,\, .
\end{equation}
Then the reduced bienergy functional \eqref{reducedbienergiatrawarped} takes the following form:
\begin{equation}\label{reducedbienergiabeta}
    E_2(\varphi_\beta) =\frac{1}{2}  \int_{\R} \,\,  \left [\,\ddot{\beta}+ (m-2)\, \dot{\beta}\, - \,(m-1) \,h (\beta) \, h' (\beta)\,\right ] ^2 \,\,e^{(m-4)t} \, \, dt \,\, .
\end{equation}
Also, we observe that, in terms of $\beta$, the conformality condition \eqref{equazionediconformalita} becomes
\begin{equation}\label{conformalitadopocambiovariabile}
    \dot{\beta}=h(\beta) \,\, .
\end{equation}
Here, and also in Section \ref{Equivariant-stability} below, we shall refer to the same rotationally symmetric map by writing $\varphi_\alpha$ or $\varphi_\beta$, the choice for $\alpha$ or $\beta$ to describe the map being the one which is more suitable to simplify a specific discussion or calculation.
Next, by using \eqref{eqeulero2}, we compute explicitly the condition of biharmonicity. Together with conformality \eqref{conformalitadopocambiovariabile} and $m=4$, that leads us again to the condition of Case 1 of the previous section, i.e.,
\begin{equation}\label{biharmonic+conformal}
    2\,h'^3-h'\,(2+h\,h'')\,-\, h^2\, h''' \,=\,0 \,\, .
\end{equation}

Now, we can state our main result:
\begin{theorem}\label{maintheorem} Suppose that, in the notation of \eqref{rotationallysymmetricmaps},
\begin{eqnarray}\label{rotationallysymmetricmapsdaerre4}
 \nonumber
    \varphi_{\alpha}\,: \,\,\, \R^4 &\to&
    \,\,M' _h\,^4 (o') \\
    ( \theta, \, r) \, &\mapsto& \, (\theta, \, \alpha(r)) \,\, .
\end{eqnarray}
is a  rotationally symmetric, proper biharmonic conformal diffeomorphism. Then $  \varphi_{\alpha}$ is either as in Case 1B or 1C of Proposition \ref{proposizionediclassificazione}.
\end{theorem}
\begin{proof}Using the conformality hypothesis \eqref{conformalitadopocambiovariabile} and $m=4$, we compute explicitly the Hamiltonian $H$ (defined in \eqref{definizionehamiltoniana}) associated to \eqref{reducedbienergiabeta}: we obtain (up to a constant)
\begin{equation}\label{hamiltonianasuerre4}
    H= h^2\, (1-h'^2+ h\,h'') \,\, .
\end{equation}
Since the Lagrangian $L$ in \eqref{reducedbienergiabeta}, when $m=4$, does not depend on $t$, according to Proposition \ref{propositionH} we conclude that $H$ is constant along any solution. Now, since \eqref{boundaryconditionforbeta} holds and $h(0)=0,\,h'(0)=1$, we deduce that the only possibility is that $H \equiv 0$ along a solution. Therefore, in the presence of a nontrivial solution (i.e., $\beta \not \equiv 0$) we deduce that necessarily
\begin{equation}\label{equazionechiave}
    1-h'^2+ h\,h''=0 \,\, .
\end{equation}
Next, derivation of \eqref{equazionechiave} along a nonconstant solution leads us to conclude that
\begin{equation}\label{equazionechiavebis}
     h\,h'''-h' \, h''=0 \,\, .
\end{equation}
Now, using the Jacobi equation \eqref{equazionecurvaturaradiale} into \eqref{equazionechiavebis}, we deduce that, along a nonconstant solution,
\begin{equation}\label{equazionechiavetris}
     h'''= -\,K\, h' \,\, .
\end{equation}
On the other hand, taking derivatives on both sides of the Jacobi equation \eqref{equazionecurvaturaradiale}, we obtain
\begin{equation}\label{equazionechiavequadris}
     h'''= -\,K\, h' - K'\,h\,\, .
\end{equation}
Finally, comparing \eqref{equazionechiavetris} and \eqref{equazionechiavequadris}, we conclude that $K' \equiv 0$, from which it follows that the solution is either as in 1B of Proposition \ref{proposizionediclassificazione} (constant positive curvature) or as in 1C (constant negative curvature).
\end{proof}
\begin{remark} In the final part of the proof, we have implicitly used the property that a model
with constant radial curvature is a space form (see \cite{GW}).
\end{remark}

\subsection{Maps from a cylinder to a model.}
We show that the use of the Hamiltonian $H$ can be useful in other related contexts.
We consider rotationally symmetric maps as in \eqref{equivariantsymmetricmaps} and assume that $f(r) \equiv 1, \, r \in \R$ . The bienergy functional has (up to a constant) the following form:
\begin{equation}\label{reducedbienergiacilindro-modello}
    E_2(\varphi_\alpha) =\frac{1}{2}  \int_{-\, \infty}^{+\infty }\,\,  \left [\ddot{\alpha} - \,\lambda\, \, h(\alpha)\,h'(\alpha) \right ]^2 \,\,\, dr \,\, .
\end{equation}
In this case the Hamiltonian $H$, computed by means of \eqref{definizionehamiltoniana}, is given by:
\begin{equation}\label{hamiltoniana-cilindri}
    H= -\, \dot{\alpha}\,\dot{\tau_\alpha}  \,+\frac{1}{2}\,\tau_\alpha\,(\,2\ddot{\alpha}\,-\,\tau_\alpha\,) \,\, ,
\end{equation}
where $\tau_\alpha$ denotes the following expression:
\begin{equation}\label{definizioneditaualfa}
    \tau_\alpha= [\,\ddot{\alpha}\,-\, \lambda\, h(\alpha)\,h'(\alpha) \,]  \,\, .
\end{equation}
Since $L=L(\alpha,\, \ddot{\alpha})$, according to Proposition \ref{propositionH} the Hamiltonian $H$ is constant along solutions. By way of example, let us consider, in \eqref{reducedbienergiacilindro-modello}, the special case of maps into a sphere (i.e., $h(\alpha)= \sin \alpha$). In \cite{MR} we pointed out that there are constant, proper solutions
\begin{equation}\label{properconstantsolutions}
    \alpha \equiv \frac{\pi}{4} \quad {\rm and} \quad  \alpha \equiv \frac{3\,\pi}{4} \,\, .
\end{equation}
Here, by using the Hamiltonian $H$, we can prove that \emph{there are no other proper solutions $\varphi_\alpha$ such that $|\tau(\varphi_\alpha)| \,\equiv \, {\rm constant}$}. Indeed, suppose that $\alpha$ gives rise to a nonconstant proper solution such that
\begin{equation}\label{constantnormatau}
    |\tau(\varphi_\alpha)|^2=[\,\ddot{\alpha}\,-\, \lambda\, \sin (\alpha)\, \cos(\alpha) \,]^2 = C^2 \quad (C> 0) \,\, .
\end{equation}
Under these assumptions, the Hamiltonian $H$ in \eqref{hamiltoniana-cilindri} becomes
\begin{equation}\label{Hamiltoniana-cilindro-sfera}
    H= \pm \, \frac{C}{2}\,(\,\ddot{\alpha} - (\pm) \,C \,) \,\, ,
\end{equation}
from which we easily deduce that $\ddot{\alpha}$ is a constant. It follows that the quantity $\sin (\alpha)\, \cos(\alpha)$ must also be a constant, a fact from which our assertion follows immediately. By the same argument (simply use either $h(\alpha)=\alpha$ or $h(\alpha)=\sinh (\alpha)$ in \eqref{reducedbienergiacilindro-modello}), we can also deduce that there exists no proper biharmonic critical point of \eqref{reducedbienergiacilindro-modello} such that $|\tau(\varphi_\alpha)| \,\equiv \, {\rm constant}$, neither in $\R^m$ nor in $H^m$.

\section{Equivariant stability}\label{Equivariant-stability} In this section we carry out the study of the second variation at the proper biharmonic  conformal diffeomorphisms which we described in Proposition \ref{proposizionediclassificazione}. More precisely, we shall study the \emph {equivariant stability} of the solutions given in \eqref{definizionesoluzionedaeuclideoasfera} and \eqref{definizionesoluzionedaeuclideoaiperbolico}. Since we are only interested in the case that the domain model is $\R^4$ (or the ball $B^4 \,\subset \,\R^4$), it is convenient to make again the change of variable \eqref{cambiovariabile euclideanspaces}. For the sake of convenience, we rewrite as follows the relevant bienergy \eqref{reducedbienergiabeta} in terms of $\beta$:
\begin{equation}\label{bienergiaperstudiodistabilita}
    E_2(\varphi_\beta) =  \frac{1}{2}\, \int \,\,\left [\ddot{\beta} + \,2\, \dot{\beta}- \,3\,q(\beta)\, \right ]^2 \, dt \,\, ,
\end{equation}
where the integral is over the domain of $\beta$ and, to simplify the calculations of this section, we have set
\begin{equation}\label{definizionediq}
    q(\beta)=h(\beta)\,h'(\beta) \,\, .
\end{equation}
Note that we shall only be concerned with the cases $h(\beta)=(1\slash d)\,\sin (d\,\beta)$ (Case 1B of Proposition \ref{proposizionediclassificazione}) and $h(\beta)=(1 \slash d)\,\sinh (d\,\beta)$ (Case 1C of Proposition \ref{proposizionediclassificazione}). For future use, we observe that, in terms of $\beta$, the explicit expressions of the proper solutions \eqref{definizionesoluzionedaeuclideoasfera} and \eqref{definizionesoluzionedaeuclideoaiperbolico} are
\begin{equation}\label{definizionesoluzionebetadaeuclideoasfera}
    \beta(t)= \,\,\frac{2}{d} \, \arctan (c^2\,e^t)\,\, ,\,\, t\,\in \, \R \,\, (c>0)
\end{equation}
and
\begin{equation}\label{definizionesoluzionebetadaeuclideoaiperbolico}
 \beta(t)= \,\, \,\,\frac{2}{d} \, \tanh ^{-1} (c^2\,e^t)\,\,,\,\,t < \ln \left (\,\frac{1}{c^2}\, \right )\,\,(c>0)
\end{equation}
respectively. We shall say that a critical point $\varphi_\beta$ of the type \eqref{rotationallysymmetricmaps} is \emph{equivariant stable} if
\begin{equation}\label{definizioneequivariantstable}
  \nabla^2 \, E_2(\varphi_\beta)\,({\mathcal V},{\mathcal V}) \,=\, \frac {d^2 \,\, E_2(\varphi_{(\beta+sV)})  }{ds^2} \Big |_{s=0} \quad > \quad  0
\end{equation}
for all smooth vector fields ${\mathcal V}=V(t)\, \partial/\partial \beta$, such that  $V(t)$ is a real valued, compactly supported function  defined over the domain of $\beta(t)$ ($V(t) \not \equiv 0)$.

In the following lemma we compute explicitly the equivariant second variation in \eqref{definizioneequivariantstable}.

\begin{lemma}\label{lemmavariazioneseconda} At the conformal, proper critical points of Proposition \ref{proposizionediclassificazione} (Cases 1B and 1C, i.e., \eqref{definizionesoluzionebetadaeuclideoasfera} and \eqref{definizionesoluzionebetadaeuclideoaiperbolico} respectively) we have:
\begin{equation}\label{variazionesecondaesplicita}
    \nabla^2 \, E_2(\varphi_\beta)\,({\mathcal V},{\mathcal V})= \int \, \left \{\left [\ddot{V} +2\,\dot{V}-3\,q'(\beta)\,V \right ]^2+6\,q''(\beta)\,h(\beta)\,(h'(\beta)-1)\, V^2 \right \}\,dt \,\, ,
\end{equation}
where the integral is over the domain of $\beta(t)$ .
\end{lemma}
\begin{proof}
First, we show that, for a generic critical point,
\begin{equation}\label{variazionesecondaesplicitaprimostep}
    \nabla^2 \, E_2(\varphi_\beta)\,({\mathcal V},{\mathcal V})= \int \,\Big \{\left [\ddot{V} +2\,\dot{V}-3\,q'(\beta)\,V \right ]^2\,-\,3\,q''(\beta)\,V^2 \, \left [ \ddot{\beta} + \,2\, \dot{\beta}- \,3\,q(\beta)\,\right ]\Big\}\,dt \,\, .
\end{equation}
To this purpose, the first step is to compute
\begin{eqnarray}\label{derivataprimanelcalcolovariazioneseconda}
&&\frac {d \,\, E_2(\varphi_{(\beta+sV)})  }{ds} =\\
&=& \frac{1}{2}\,\int \,\left \{\frac{d}{ds} \,\left [\ddot{\beta}+s\ddot{V} +2(\dot{\beta}+s\dot{V})-3\,q(\beta+sV) \right ]^2\, \right \}\,dt\nonumber \\ \nonumber
 &=&\int\, \left [\ddot{\beta}+s\ddot{V} +2(\dot{\beta}+s\dot{V})-3\,q(\beta+sV) \right ] \cdot \left [\ddot{V} +2\dot{V}-3\,q'(\beta+sV)\,V \right ] \, dt \,\,.
\end{eqnarray}
Next, using \eqref{derivataprimanelcalcolovariazioneseconda}:
\begin{eqnarray}
&&  \frac {d^2 \,\, E_2(\varphi_{(\beta+sV)})  }{ds^2} \Big |_{s=0} =\\
&=&   \int\, \left \{ \frac{d}{ds}\left [ \left (\ddot{\beta}+s\ddot{V} +2(\dot{\beta}+s\dot{V})-3\,q(\beta+sV) \right )
\left(\ddot{V} +2\dot{V}-3\,q'(\beta+sV)\,V \right )\right ] \right \}\Big |_{s=0} \, dt \nonumber\\ \nonumber
   &=& \int \,\left\{\left [\ddot{V} +2\,\dot{V}-3\,q'(\beta)\,V \right ]^2\,-\,3\,q''(\beta)\,V^2 \, \left [ \ddot{\beta} + \,2\, \dot{\beta}- \,3\,q(\beta)\,\right ]\right\}\,dt \,,
\end{eqnarray}
as required to prove \eqref{variazionesecondaesplicitaprimostep}. Finally, we use the fact that our critical points are conformal and so they verify \eqref{conformalitadopocambiovariabile}. Taking derivatives, from \eqref{conformalitadopocambiovariabile} we deduce
\begin{equation}\label{conseguenzadellaconformalitasubetasecondo}
    \ddot{\beta}=h'(\beta)\,\dot{\beta}=h'(\beta)\,h(\beta)=q(\beta)\,\,.
\end{equation}
Now, using \eqref{conformalitadopocambiovariabile} and \eqref{conseguenzadellaconformalitasubetasecondo} into \eqref{variazionesecondaesplicitaprimostep} we obtain \eqref{variazionesecondaesplicita}, as required to end the proof of the lemma.
\end{proof}

\begin{remark}
Integrating by parts \eqref{variazionesecondaesplicitaprimostep} and using the fact that $V(t)$ is a compactly supported function we obtain
$$
 \nabla^2 \, E_2(\varphi_\beta)\,({\mathcal V},{\mathcal V})= \int \, \langle I({\mathcal V}), {\mathcal V}\rangle\, dt\,\,,
$$
where
\begin{equation}\label{eq:second-variation-operator}
I({\mathcal V})=\Big\{ \ddddot{V}-\big[4+6 q'(\beta)\big]\ddot{V}+\big[9 q'^2(\beta)-3q''(\beta)\big(\ddot{\beta}-3q(\beta)\big)\big]V\Big\}\frac{\partial}{\partial \beta}\,\,.
\end{equation}
Direct computation shows that the operator $I({\mathcal V})$ defined in \eqref{eq:second-variation-operator} can be derived  by using the second variation formula obtained by Jiang in \cite{Jiang, Jiang-eng} (see also \cite{Oniciuc}).
\end{remark}
We are now in the right position to state our results in this context.
\begin{theorem}\label{teoremavariazionesecondaiperbolico} Let
$$
\varphi_\beta\,: \,B^4 (1 \slash c^2)\,\rightarrow\,H^4(-d^2)
$$
be the  rotationally symmetric, proper biharmonic conformal diffeomorphism defined by means of the function $\beta$ in \eqref{definizionesoluzionebetadaeuclideoaiperbolico} . Then $\varphi_\beta\,$ is equivariant stable.
\end{theorem}
\begin{proof}By using $h(\beta)=(1 \slash d) \, \sinh (d\, \beta)$ into \eqref{definizionediq} and \eqref{variazionesecondaesplicita}, we find that, in this case, the second variation takes the following form:
\begin{eqnarray}\label{variazionesecondacasoiperbolico}
    &&\quad      \nabla^2 \, E_2(\varphi_\beta)\,({\mathcal V},{\mathcal V})= \\ \nonumber
    &=& \int_{-\infty}^{\ln (\frac{1}{c^2})}  \left \{\left [\ddot{V} +2\,\dot{V}-3\cosh (2d\,\beta)\,V \right ]^2 +  24\,\cosh(d\,\beta)\,\sinh^2 (d\,\beta)\left ( \cosh (d\,\beta)-1 \right )\, V^2\right \}\,dt
\end{eqnarray}
from which the conclusion of the proof is immediate.
\end{proof}
In a context of Dirichlet's problem, we find that our conformal, proper biharmonic diffeomorphisms into $S^4$ are equivariant stable if we restrict them in such a way that only the closed, upper hemisphere of $S^4$ is covered. More precisely, we have
\begin{theorem}\label{teoremavariazionesecondaDirichlet} Let us consider proper biharmonic solutions of the boundary value problem
\begin{equation}\label{problemadidirichlet}
    \left \{\begin{array}{c}
              \varphi_\beta\,: \,\overline{B^4 (1 \slash c^2)}\,\rightarrow\,S^4(d^2) \\
              \,\\
              \beta(\ln (1 \slash c^2))=(\pi \slash (2\,d)) \,\, .
            \end{array}
\right .
\end{equation}
The conformal, rotationally symmetric, proper biharmonic solution $\varphi_\beta$, defined by means of the function $\beta$ in \eqref{definizionesoluzionebetadaeuclideoasfera} with $t \leq \ln (1 \slash c^2)$, is equivariant stable with respect to variations which preserve the boundary data.
\end{theorem}
\begin{proof} Formally, we have to consider the equivariant second variation \eqref{variazionesecondaesplicita} with respect to compactly supported vector fields ${\mathcal V}=V(t)\, \partial/\partial \beta$ which, in order to preserve the boundary data, satisfy the additional condition
$$V(\ln (1 \slash c^2)) = \dot{V}(\ln (1 \slash c^2)) =0\,\,.$$
By using $h(\beta)=(1 \slash d) \, \sin (d\, \beta)$ into \eqref{definizionediq} and \eqref{variazionesecondaesplicita}, we find that, in this case, the second variation takes the following form:
\begin{eqnarray}\label{variazionesecondacasosfera}
  && \quad  \nabla^2 \, E_2(\varphi_\beta)\,({\mathcal V},{\mathcal V})= \\ \nonumber
    && \int_{-\infty}^{\ln (1\slash c^2)}  \left \{\left [\ddot{V} +2\,\dot{V}-3\cos (2d\,\beta)\,V \right ]^2 + 24\,\cos (d\,\beta)\,\sin^2 (d\,\beta)\,\left (1- \cos (d\,\beta)\,\right )\, V^2 \right \}\,dt \,\, .
\end{eqnarray}
Now, since $ 0 < \beta \leq \big (\pi \slash (2d) \big )$ when $t \leq \ln (1 \slash c^2)$, the conclusion of the proof follows easily from \eqref{variazionesecondacasosfera}.
It should be noted that in this case the map $\varphi_{\beta}$ covers exactly half sphere minus the pole.
\end{proof}
\begin{remark}
The methods of this section can be used to check that the proper biharmonic conformal diffeomorphisms described in A and C of Remark~\ref{re:referee} are equivariant stable. Similarly, the proper biharmonic examples described in B of Remark~\ref{re:referee}, restricted to $\R^4\setminus\{B^4 (c^2)\}$, are equivariant stable with respect to variations which preserve the boundary data (i.e., $\beta(\ln c^2)=\pi/ (2d)$).
\end{remark}

\begin{remark}\label{errata-corrige-MJM} We take this opportunity to point out that, in a similar stability context, formula $(3.9)$ in \cite{MR} should be replaced by
\begin{equation*}
\nabla^2 \, E_2^{\varphi}(\alpha^*)\,(V,V) = \int_0^{2 \pi} \, \, \left [ 2 \, \ddot{V}^2 - 2\, V^2 \,k^4 \,\right ] \, d \theta \,\, ,
\end{equation*}
from which it follows easily that the critical points of Theorem 3.2 in \cite{MR} are \emph{unstable}, and Remark 3.3 of \cite{MR} should then be deleted.
\end{remark}

\end{document}